\documentclass[11pt]{article}

\usepackage{amsmath,amssymb,amsthm,amscd,epsfig}
\newcommand{\vp}{\varepsilon}

\theoremstyle{plain}

\newtheorem{thm}{Theorem}

\newtheorem*{unthm}{Theorem}

\newtheorem{cor}{Corollary}

\newtheorem{pro}{Proposition}
\newtheorem*{unpro}{Proposition}

\theoremstyle{definition}

\newtheorem{defn}{Definition}

\theoremstyle{remark}

\begin{document}

\title{On a class of stable conditional measures}

\author{Eugen Mihailescu}

\date{}
\maketitle

\begin{abstract}
The dynamics of endomorphisms (smooth non-invertible maps)
presents many differences from that of diffeomorphisms or that of
expanding maps;  most methods from those cases do not work if the
map has a basic set of saddle type with self-intersections. In
this paper we study the conditional measures of a certain class of
equilibrium measures, corresponding to a measurable partition
subordinated to local stable manifolds. We show that these
conditional measures are geometric probabilities on the local
stable manifolds, thus answering in particular the questions
related to the stable pointwise Hausdorff and box dimensions.
These stable conditional measures are shown to be absolutely
continuous if and only if the respective basic set is a
non-invertible repellor. \ We find also invariant measures of
maximal stable dimension, on folded basic sets. Examples are given
too, for such non-reversible systems.
\end{abstract}

\textbf{MSC 2000:} Primary: 37D35, 37B25, 34C45. Secondary: 37D20.

\textbf{Keywords:} Equilibrium measures for hyperbolic
non-invertible maps, stable manifolds, conditional measures,
folded repellors, pointwise dimensions of measures.

\section{Background and outline of the paper.}

 In this paper we will study non-invertible smooth (say
 $\mathcal{C}^2$) maps on a Riemannian manifold $M$, called \textit{endomorphisms},
which are uniformly hyperbolic on a basic set $\Lambda$. Here by
\textit{basic set} for an endomorphism $f:M \to M$, we understand
a compact topologically-transitive set $\Lambda$, which has a
neighbourhood $U$ such that $\Lambda = \mathop{\cap}\limits_{n \in
\mathbb Z} f^n(U)$.

Considering non-invertible transformations makes sense from the
point of view of applications, since the evolution of a
non-reversible physical system is usually given by a
time-dependent differential equation $\frac{dx(t)}{dt} = F(x(t))$
whose solution, the flow $(f^t)_t$, may not consist necessarily of
diffeomorphisms. However if we look at the ergodic (qualitative)
properties of the associated flow (equilibrium measures, Lyapunov
exponents, conditional measures associated to measurable
partitions), we may replace it with a discrete non-invertible
dynamical system (\cite{ER}). The theory of hyperbolic
diffeomorphisms (Axiom A) has been studied by many authors (see
for example \cite{Bo}, \cite{ER}, \cite{KH}, \cite{Ru-carte}, and
the references therein); also the theory of expanding maps was
studied extensively (see for instance \cite{Ru-exp}), and the fact
that the local inverse iterates are contracting on small balls is
crucial in that case.

 However, the
theory of smooth non-invertible maps which have saddle basic sets,
is significantly different from the two above mentioned cases.
Most methods of proof from diffeomorphisms or expanding maps, do
not work here due to the complicated \textbf{overlappings and
foldings} that the endomorphism may have in the
 basic set $\Lambda$. The unstable manifolds depend in general on the \textbf{choice of a sequence} of
consecutive preimages, not only on the initial point (as in the
case of
 diffeomorphisms). So the unstable manifolds do not form a
 foliation, instead they may intersect each other both inside and
 outside $\Lambda$.
Moreover the local inverse iterates do not contract necessarily on
small balls, instead they will grow exponentially (at least for
some time) in the stable direction. Also, an arbitrary basic set
$\Lambda$ is not necessarily totally invariant for $f$, and there
do not always exist Markov partitions on $\Lambda$. We mention
also that endomorphisms on Lebesgue spaces behave differently than
invertible transformations even from the point of view of
classifications in ergodic theory, see \cite{PW}.

 We will work in the sequel with a hyperbolic endomorphism $f$ on
a basic set $\Lambda$; such a set is also called a \textit{folded
basic set} (or a basic set with \textit{self-intersections}). By
\textit{$n$-preimage} of a point $x$ we mean a point $y$ such that
$f^n(y) = x$.  By \textit{prehistory} of $x$ we understand a
sequence of consecutive preimages of $x$, belonging to $\Lambda$,
and denoted by $\hat x = (x, x_{-1}, x_{-2}, \ldots)$ where
$f(x_{-n}) = x_{-n+1}, n
>0$, with $x_0 = x$. And by \textit{inverse limit} of $(f, \Lambda)$ we mean the space of
all such prehistories, denoted by $\hat \Lambda$. For more about
these aspects, see \cite{Ro}, \cite{M-DCDS06}. By the definition
of a basic set $\Lambda$, we assume that $f$ is
\textit{topologically transitive} on $\Lambda$ as an endomorphism,
i.e that there exists a point in $\Lambda$ whose iterates are
dense in $\Lambda$.

\textit{Hyperbolicity} is defined for endomorphisms (see
\cite{Ru-carte}) similarly as for diffeomorphisms, with the
crucial difference that now the unstable spaces (and thus the
local unstable manifolds) depend on whole prehistories; so we have
the stable tangent spaces $E^s_x, x \in \Lambda$, the unstable
tangent spaces $E^u_{\hat x}, \hat x \in \hat \Lambda$, the
\textit{local stable manifolds} $W^s_r(x), x \in \Lambda$ and the
\textit{local unstable manifolds} $W^u_r(\hat x), \hat x \in \hat
\Lambda$. As there may be (infinitely) many unstable manifolds
going through a point, we do not have here a well defined holonomy
map between stable manifolds, by contrast to the diffeomorphism
case. For more details on endomorphisms, see \cite{Ru-carte},
\cite{PDL}, \cite{M-DCDS06}, \cite{MU-CJM}, etc.

\begin{defn}\label{stable}
Consider a smooth (say $\mathcal{C}^2$) non-invertible map $f$
which is hyperbolic on the basic set $\Lambda$, such that the
critical set of $f$ does not intersect $\Lambda$. Define the
\textit{stable potential} of $f$ as $\Phi^s(y):= \log |Df_s(y)|, y
\in \Lambda$. By \textit{stable dimension} (at a point $x \in
\Lambda$) we understand the Hausdorff dimension $\delta^s(x):=
HD(W^s_r(x) \cap \Lambda)$. We will also say that $f$ is
\textit{c-hyperbolic} on $\Lambda$ if $f$ is hyperbolic on
$\Lambda$, there are no critical points of $f$ in $\Lambda$ and
$f$ is conformal on the local stable manifolds.
\end{defn}

The relations between thermodynamic formalism and the dynamics of
diffeomorphisms or expanding maps form a rich field (see for
instance \cite{Ba}, \cite{Bo}, \cite{ER}, \cite{LY},
\cite{Ru-exp}, etc.) And in \cite{M-Cam}, \cite{MU-CJM},
\cite{MU-DCDS08}, we studied some aspects of the thermodynamic
formalism for non-invertible smooth maps.

\textbf{Examples} of hyperbolic endomorphisms are numerous, for
instance hyperbolic solenoids and horseshoes with
self-intersections (\cite{Bot}), polynomial maps in higher
dimension hyperbolic on certain basic sets, skew products with
overlaps in their fibers (\cite{MU-DCDS08}), hyperbolic toral
endomorphisms or perturbations of these, etc. $\hfill\square$

In this non-invertible setting, a special importance is presented
by \textit{constant-to-one endomorphisms}. For such endomorphisms,
we study the family of conditional measures of a certain
equilibrium measure, family associated to a measurable partition
subordinated to local stable manifolds.

If a \textbf{topological condition} is satisfied, namely if the
number of preimages remaining in $\Lambda$ is constant along
$\Lambda$, we showed in \cite{MU-CJM} the following:

\begin{unthm}[Independence of stable dimension]
If the endomorphism $f$ is c-hyperbolic on the basic set $\Lambda$
(see Definition \ref{stable}) and if the number of $f$-preimages
of any point from $\Lambda$, remaining in $\Lambda$ is constant
and equal to $d$, then the stable dimension $\delta^s(x)$ is equal
to the unique zero $t^s_d$ of the pressure function $t \to
P(t\Phi^s - \log d)$, for any $x \in \Lambda$. The common value of
the stable dimension along $\Lambda$ will be denoted by
$\delta^s$.
\end{unthm}

In fact if $f$ is \textbf{open} on $\Lambda$, we proved (see
\cite{MU-CJM}, and Proposition 1 of \cite{M-Cam}) the following:

\begin{unpro}[\cite{M-Cam}, \cite{MU-CJM}]
Let an endomorphism $f:M \to M$ which has a basic set $\Lambda$,
disjoint from the critical set of $f$. Assume that $\Lambda$ is
connected and $f|_\Lambda: \Lambda \to \Lambda$ is open. Then the
cardinality of the set $f^{-1}(x) \cap \Lambda$ is constant, when
$x$ ranges  in $\Lambda$.
\end{unpro}

Examples of hyperbolic open endomorphisms on saddle sets are given
in the end of the paper.

\begin{defn}\label{stable-measure}
Let an endomorphism $f$ c-hyperbolic on the basic set $\Lambda$,
such that the number of $f$-preimages of any point from $\Lambda$,
remaining in $\Lambda$, is constant and equal to $d$. Then we call
the equilibrium measure of  $\delta^s \cdot \Phi^s$, the
\textit{stable equilibrium measure} of $f$ on $\Lambda$, and
denote it by $\mu_s$.
\end{defn}

We notice that, since the stable foliation is Lipschitz continuous
for endomorphisms (see \cite{MU-CJM}), the potential $\delta^s
\cdot \Phi^s$ is Holder continuous; thus it can be shown by
lifting the measure to the inverse limit $\hat \Lambda$, that
there exists a unique equilibrium measure $\mu_s$ of $\delta^s
\cdot \Phi^s$ (we can apply the results for homeomorphisms from
\cite{KH} on the inverse limit $\hat \Lambda$, in order to get the
uniqueness).

We will show in Theorem \ref{main} that if the number of
$f$-preimages in $\Lambda$ is constant, then the
\textbf{conditional measures} of $\mu_s$ associated to a
measurable partition subordinated to the local stable manifolds,
are \textbf{geometric probabilities} of exponent $\delta^s$. This
will answer then in Corollary \ref{pointwise} the question of the
pointwise Hausdorff dimension and the pointwise box dimension of
the equilibrium measure $\mu_s$ on local stable manifolds (see for
instance \cite{Ba} for definitions). In the constant-to-1
non-invertible case, we show in particular in Corollary
\ref{maximal} that these stable conditional measures are measures
of \textbf{maximal dimension} (in the sense of \cite{BW}) on the
intersections of local stable manifolds with the folded basic set
$\Lambda$.

Our approach will be different both from the case of
diffeomorphisms and from that of expanding maps. In Proposition
\ref{pieces} (which is the main ingredient for the proof of
Theorem \ref{main}), we compare the equilibrium measure on various
different components of the preimage set of a small "cylinder"
around an unstable manifold. We will have to carefully estimate
the equilibrium measure $\mu_s$ on the different pieces of the
iterates of Bowen balls, in order to get good estimates for the
cylinders around local unstable manifolds, $B(W^u_r(\hat x),
\vp)$. This will be done by a process of \textbf{desintegrating}
the measure on the various components of the preimages of borelian
sets, and then by successive re-combinations. Thus we will
reobtain the measure $\mu_s$ on an arbitrary open set, and then
will use the essential uniqueness of the family of conditional
measures of $\mu_s$; for background on conditional measures
associated to measurable partitions on Lebesgue spaces, see
\cite{Ro}.

In Corollary \ref{absolute} we prove that the conditional measures
of $\mu_s$ on the local stable manifolds over $\Lambda$ are
\textbf{absolutely continuous} if and only if, the stable
dimension is equal to the real dimension of the stable tangent
space $\text{dim} E^s_x$; \ and we show that this is equivalent to
$\Lambda$ being a folded repellor.

 We will also give in the end Examples of hyperbolic
constant-to-1 \textbf{folded basic sets} for which Theorem
\ref{main} and its Corollaries do apply. In particular we provide
examples of folded repellors obtained for perturbation
endomorphisms, which are not Anosov and for which we prove the
absolute continuity of the stable conditional measures on their
(non-linear) stable manifolds.

\section{Main proofs and applications.}

For our first result we assume only that $f$ is a smooth
endomorphism which is hyperbolic on a basic set $\Lambda$. We will
give a comparison between the values of an arbitrary equilibrium
measure $\mu_\phi$ (corresponding to a Holder continuous potential
$\phi$ on $\Lambda$) on the different pieces/components of the
preimages of a borelian set; this will be useful when we will
estimate later on, the measure $\mu_s$ on certain sets.

By a \textit{Bowen ball} $B_n(x, \vp)$ we understand the set $\{y
\in \Lambda, d(f^iy, f^i x) < \vp, i = 0, \ldots, n\}$, for $x \in
\Lambda$ and $n > 1$. If $\phi$ is a continuous real function on
$\Lambda$ and $m$ is a positive integer, we denote by $S_m
\phi(y):= \phi(y) + \phi(f(y)) + \ldots + \phi(f^m(y))$ the
consecutive sum of $\phi$ on the $n$-orbit of $y \in \Lambda$. And
by $P(\phi)$ we denote the \textit{topological pressure} of the
potential $\phi$ with respect to the function $f|_\Lambda$.

\begin{pro}\label{pieces}

Let $f$ be an endomorphism, hyperbolic on a basic set $\Lambda$;
consider also a Holder continuous potential $\phi$ on $\Lambda$
and $\mu_\phi$ be the unique equilibrium measure of $\phi$. Let a
small $\vp>0$, two disjoint Bowen balls $B_k(y_1, \vp), B_m(y_2,
\vp)$ and a borelian set $A \subset f^k(B_k(y_1, \vp)) \cap
f^m(B_m(y_2, \vp))$, s.t $\mu_\phi(A) >0$; denote by $A_1:=
f^{-k}A \cap B_k(y_1, \vp), A_2 := f^{-m}A \cap B_m(y_2, \vp)$.
Then there exists a positive constant $C_\vp$ independent of $k,
m, y_1, y_2$ such that $$ \frac{1}{C_\vp} \mu_\phi(A_2) \cdot
\frac{e^{S_k \phi(y_1)}}{e^{S_m\phi(y_2)}} \cdot P(\phi)^{m-k} \le
\mu_\phi(A_1) \le C_\vp \mu_\phi(A_2) \cdot
\frac{e^{S_k\phi(y_1)}}{e^{S_m\phi(y_2)}} \cdot P(\phi)^{m-k}$$
\end{pro}

\begin{proof}
 Let us fix a Holder potential $\phi$. We will denote the
equilibrium measure $\mu_\phi$ by $\mu$ to simplify notation. We
will work with $f$ restricted to $\Lambda$.

As in \cite{KH}, since the borelian sets with boundaries of
$\mu$-measure zero form a sufficient collection, we will assume
that each of the sets $A_1, A_2$ have boundaries of $\mu$-measure
zero.

From construction $f^k(A_1) = f^m(A_2)$, and assume for example
that $m \ge k$. Now the equilibrium measure $\mu$ can be
considered as the limit of the sequence of measures (see
\cite{KH}):

$$\tilde \mu_n:= \frac {1}{P(f, \phi, n)} \cdot
\mathop{\sum}\limits_{x \in \text{Fix}(f^n)} e^{S_n\phi(x)}
\delta_x,$$ where $P(f, \phi, n):= \mathop{\sum}\limits_{x \in
\text{Fix}(f^n)} e^{S_n\phi(x)}, n \ge 1$.

So we have
\begin{equation}\label{ai}
\tilde \mu_n(A_1) = \frac {1}{P(f, \phi, n)} \cdot
\mathop{\sum}\limits_{x \in \text{Fix}(f^n) \cap A_1}
e^{S_n\phi(x)}, n \ge 1
\end{equation}

Let us consider now a periodic point $x \in \text{Fix}(f^n) \cap
A_1$; by definition of $A_1$, it follows that $f^k(x) \in A$, so
there exists a point $y \in A_2$ such that $f^m(y) = f^k(x)$.
However the point $y$ does not have to be periodic.

Now we will use the Specification Property (\cite{Bo}, \cite{KH})
on the hyperbolic compact locally maximal set $\Lambda$:  if
$\vp>0$ is fixed, then there exists a constant $M_\vp>0$ such that
for all $n
>> M_\vp$, there exists a $z \in \text{Fix}(f^{n+m-k})$ s.t $z$ $\vp$-shadows
the $(n+m-k-M_\vp)$-orbit of $y$.

Let now $V$ be an arbitrary neighbourhood of the set $A_2$ s.t $V
\subset B_m(y_2, \vp)$. Consider two points $x, \tilde x \in
\text{Fix}(f^n) \cap A_1$ and assume the same periodic point $z
\in V \cap \text{Fix}(f^{n+m-k})$ corresponds to both $x$ and
$\tilde x$ by the above procedure. This means that the
$(n-k-M_\vp)$-orbit of $f^m z$, $\vp$-shadows the
$(n-k-M_\vp)$-orbit of $f^k x$ and also the $(n-k-M_\vp)$-orbit of
$f^k \tilde x$. Hence the $(n-M_\vp-k)$-orbit of $f^k x$,
$2\vp$-shadows the $(n-M_\vp-k)$-orbit of $f^k \tilde x$. But
recall that we chose $x, \tilde x \in A_1 \subset B_k(y_1, \vp)$,
hence $\tilde x \in B_{n-M_\vp}(x, 2\vp)$.

Now we can split the set $B_{n-M_\vp}(x, 2\vp)$ in at most $N_\vp$
smaller Bowen ball of type $B_n(\zeta, 2\vp)$. In each of these
$(n, 2\vp)$-Bowen balls $B_n(\zeta, 2\vp)$ we may have at most one
fixed point for $f^n$. This holds since fixed points for $f^n$ are
solutions to the equation $f^n \xi = \xi$ and, on tangent spaces
we have that $Df^n - Id$ is a linear map without eigenvalues of
absolute value 1. Thus if $d(f^i \xi, f^i \zeta) < 2\vp, i = 0,
\ldots, n$ and if $\vp$ is small enough, we can apply the Inverse
Function Theorem at each step. Therefore there exists only one
fixed point for $f^n$ in each Bowen ball $B_n(\zeta, 2\vp)$. Hence
there exist at most $N_\vp$ periodic points from $\text{Fix}(f^n)
\cap \Lambda$ having the same periodic point $z \in V$ attached to
them by the above procedure.

Let us notice also that, if $x, \tilde x$ have the same point
$z\in V \cap \text{Fix}(f^{n+m-k})$ attached to them, then as
before, $\tilde x \in B_{n-M_\vp}(x, 2\vp)$. So the distances
between iterates are growing exponentially in the unstable
direction, and decrease exponentially in the stable direction.
Thus we can use the Holder continuity of $\phi$ and a Bounded
Distortion Lemma to prove that: $$|S_n\phi(x) - S_n\phi(\tilde x)|
\le \tilde C_\vp,$$ for some positive constant $\tilde C_\vp$
depending on $\phi$ (but independent of $n, x$). This can be used
then in the estimate for $\tilde \mu_n(A_1)$, according to
(\ref{ai}). We use the fact that if $z \in B_{n+m-k-M_\vp}(y,
\vp)$, then $f^m(z) \in B_{n-M_\vp-k}(f^m y, \vp)$; also recall
that $f^kx = f^m y$, so $f^m z \in B_{n-M_\vp-k}(f^k x, \vp)$.
Then from the Holder continuity of $\phi$ and the fact that $x \in
A_1 \subset B_m(y_1, \vp)$, it follows again by a Bounded
Distortion Lemma that there exists a constant $\tilde C_\vp$
(denoted as before without loss of generality) satisfying:
\begin{equation}\label{distortion}
|S_{n+m-k} \phi(z) - S_n\phi(x)| \le |S_k\phi(y_1) - S_m\phi(y_2)|
+ \tilde C_\vp,
\end{equation}
 for $n > n(\vp, m)$.

But from Proposition 20.3.3 of \cite{KH} (which extends
immediately to endomorphisms), we have that there exists a
positive constant $c_\vp$ such that for sufficiently large $n$: $$
\frac{1}{c_\vp}e^{nP(\phi)} \le P(f, \phi, n) \le c_\vp
e^{nP(\phi)}, $$ where the expression $P(f, \phi, n)$ was defined
immediately before (\ref{ai}). Hence in our case, if $n > n(\vp,
m)$ we obtain:
\begin{equation}\label{PX}
\frac{1}{c_\vp}e^{(n+m-k)P(\phi)} \le P(f, \phi, n+m-k) \le c_\vp
e^{(n+m-k)P(\phi)}, \text{and} \ \frac{1}{c_\vp} e^{nP(\phi)} \le
P(f, \phi, n) \le c_\vp e^{nP(\phi)}
\end{equation}

Recall also that there are at most $N_\vp$ points $x \in
\text{Fix}(f^n)$ which have the same attached $z \in V \cap
\text{Fix}(f^n)$. Therefore, by using (\ref{ai}),
(\ref{distortion}) and (\ref{PX}) we can infer that there exists a
constant $C_\vp>0$ such that for $n$ large enough ($n > n(\vp,
m)$),

\begin{equation}\label{ineq1}
\tilde \mu_n(A_1) \le C_\vp \tilde \mu_{n+m-k}(V) \cdot
\frac{e^{S_k \phi(y_1)}}{e^{S_m\phi(y_2)}} \cdot P(\phi)^{m-k},
\end{equation}

where we recall that $A_1 \subset B_m(y_1, \vp), A_2 \subset
B_m(y_2, \vp)$. But since $\partial A_1, \partial A_2$ have
$\mu$-measure zero, we obtain: $$ \mu(A_1) \le C_\vp \mu(V)
\frac{e^{S_k\phi(y_1)}}{e^{S_m\phi(y_2)}} \cdot P(\phi)^{m-k}$$

But $V$ has been chosen arbitrarily as a neighbourhood of $A_2$,
hence $$ \mu(A_1) \le C_\vp \mu(A_2)
\frac{e^{S_k\phi(y_1)}}{e^{S_m\phi(y_2)}} P(\phi)^{m-k} $$
Similarly we prove also the other inequality, hence we are done.

\end{proof}

Let us recall a few notions about measurable partitions (see
\cite{Ro}). Let $\zeta$ be a partition of a Lebesgue space $(X,
\mathcal{B}, \mu)$ with $\mathcal{B}$-measurable sets. Subsets of
$X$ that are unions of elements of $\zeta$ are called
$\zeta$-sets. For an arbitrary point $x\in X$ (modulo $\mu$), we
denote the unique set which contains $x$, by $\zeta(x)$. \ By
\textit{basis} for $\zeta$ we understand a countable collection
$\{B_\alpha, \alpha \in A\}$ of measurable $\zeta$-sets so that
for any two elements $C, C' \in \zeta$, there exists some $\alpha
\in A$ with $C \subset B_\alpha, C' \cap B_\alpha = \emptyset$ or
viceversa, i.e $C \cap B_\alpha = \emptyset, C' \subset B_\alpha$.
A partition $\zeta$ is called \textit{measurable} if it has a
basis as above.

Now we remind briefly the notion of \textit{family of conditional
measures} associated to a measurable partition $\zeta$. Assume we
have an endomorphism $f$ on a compact set $\Lambda$, and let a
probability borelian measure $\mu$ on $\Lambda$ which is
$f$-invariant. If $\zeta$ is a measurable partition of $(\Lambda,
\mathcal{B}, \mu)$ denote by $(\Lambda/\zeta, \mu_\zeta)$ the
\textit{factor space} of $\Lambda$ relative to $\zeta$. Then we
can attach an \textbf{essentially unique} collection of
\textit{conditional measures} $\{\mu_C\}_{C \in \zeta}$ satisfying
two conditions (see \cite{Ro}):

 \ i) $(C, \mu_C)$ is a Lebesgue space

 \ ii) for any measurable set $B \subset \Lambda$, the set $B \cap C$ is measurable in $C$ for $\mu_\zeta$-almost
 all points $C \in \Lambda/\zeta$, the function $C \to \mu_C(B \cap C)$ is
measurable on $\Lambda/\zeta$ and
 $\mu(B) = \int_{\Lambda/\zeta} \mu_C(B \cap C) d \mu_\zeta(C)$.

\begin{defn}\label{sub}
If $f$ is a hyperbolic map on a basic set $\Lambda$ and if $\mu$
is an $f$-invariant borelian measure on $\Lambda$, then a
measurable partition $\zeta$ of $(\Lambda, \mathcal{B}(\Lambda),
\mu)$ is said to be \textbf{subordinated to the local stable
manifolds} if for $\mu$-a. e $x \in \Lambda$, we have $\zeta(x)
\subset W^s_{loc}(x)$, and $\zeta(x)$ contains an open
neighbourhood of $x$ in $W^s_{loc}(x)$ (with respect to the
topology induced on the local stable manifold).
\end{defn}

Let us fix an $f$-invariant borelian measure $\mu$ on $\Lambda$.
Since we work with a uniformly hyperbolic endomorphism, we can
\textbf{construct a measurable partition $\xi$} (w. r. t $\mu$)
subordinated to the local stable manifolds, in the following way:
first we know that there is a small $r_0>0$ s. t for each $x \in
\Lambda$ there exists a local stable manifold $W^s_{r_0}(x)$. Then
it is possible to take a countable partition $\mathcal{P}$ of
$\Lambda$ (modulo $\mu$) with open sets, each having diameter less
than $r_0$ and such that the boundary of each set from
$\mathcal{P}$ has $\mu$-measure zero (see for example \cite{KH}).
Now for every open set $U \in \mathcal{P}$, and $x \in U \subset
\Lambda$, we consider the intersection between $U$ and the unique
local stable manifold going through $x$; denote this intersection
by $\xi(x)$. It is clear that $\xi(x) = \xi(y)$ if and only if
both $x, y$ are in the same set $U \in \mathcal{P}$ and they are
on the same local stable manifold $W^s_{r_0}(z)$ for some $z \in
\Lambda$. Now take the collection $\xi$ of all the borelian sets
$\xi(x), x \in U, U \in \mathcal{P}$. We see easily that $\xi$ is
a partition of $\Lambda$ (modulo sets of $\mu$-measure zero) and
that $\xi$ is measurable, since $\mathcal{P}$ was assumed
countable and, inside each member $U \in \mathcal{P}$, we can
separate any two local stable manifolds with the help of a
countable collection of $\xi$-sets (which are neighbourhoods of
local stable manifolds).  \ Therefore we have concluded the
construction of the measurable partition $\xi$ which is
subordinated to the local stable manifolds. Modulo a set of
$\mu$-measure zero we have thus a partition with pieces of local
stable manifolds, $\xi(x) \subset W^s_{r(y(x))}(y(x)), x \in
\Lambda$. In fact without loss of generality, we may assume that
for each member $A \in \xi$, there exists some $x(A) \in \Lambda$
and $r(A)\in (0, r_0)$ so that $W^s_{r(A)/2}(x(A)) \cap \Lambda
\subset A \subset W^s_{r(A)}(x(A)) \cap \Lambda$.

\textbf{Remark 1:} From the construction above it follows that,
outside a set of $\mu$-measure zero, the radius $r(A)$ can be
taken to vary continuously, i.e there exists a constant $\chi>0$
s. t for each $x$ in a set of full $\mu$-measure in $\Lambda$,
there exists a neighbourhood $U(x)$ of $x$ with
$\frac{r(\xi(z))}{r(\xi(z'))} \le \chi, z, z' \in U(x)$.

$\hfill\square$

\textbf{Notation:} In our uniformly hyperbolic setting, with the
partition $\xi$ constructed above, we denote the conditional
measure $\mu_A$ by $\mu_{A}^s$, for $W^s_{r(A)/2}(x(A)) \cap
\Lambda \subset A \subset W^s_{r(A)}(x(A)) \cap \Lambda, A \in
\xi$. We will also denote the set of centers $\{x(A), A \in \xi\}$
by $S$. In particular, if $\mu = \mu_s$, we denote the conditional
measures by $\mu^s_{s, A}$ for $A \in \xi$, or by $\mu^s_{s, x}$
when $\xi(x) = A$ for $\mu_s$-a.e $x \in \Lambda$. $\hfill\square$

Now, if $f$ is a $d$-to-1 c-hyperbolic endomorphism on the basic
set $\Lambda$, we showed in \cite{MU-CJM} that the stable
dimension $\delta^s(x)$ at \textbf{any point} $x \in \Lambda$ is
independent of $x$, and is equal to the unique zero of the
pressure function $t \to P(t \Phi^s - \log d)$. Thus we can talk
in this case about the \textbf{stable dimension of $\Lambda$} and
will denote it by $\delta^s$.

\begin{thm}\label{main}
Let $f$ be a smooth endomorphism on a Riemannian manifold $M$, and
assume that $f$ is c-hyperbolic on a basic set of saddle type
$\Lambda$. Let us assume moreover that $f$ is $d$-to-1 on
$\Lambda$. Assume that $\Phi^s(y):= \log |Df_s(y)|, y \in
\Lambda$, that $\delta^s$ is the stable dimension of $\Lambda$,
and that $\mu_s$ is the equilibrium measure of the potential
$\delta^s \Phi^s$ on $\Lambda$. Then the conditional measures of
$\mu_s$ associated to the partition $\xi$, namely $\mu_{s, A}^s$,
are geometric probabilities, i.e for every set $A \in \xi$ there
exists a positive constant $C_A$ such that $$C_A^{-1}
\rho^{\delta^s} \le \mu_{s, A}^s(B(y, \rho)) \le C_A
\rho^{\delta^s}, y \in A \cap \Lambda, 0 < \rho < \frac{r(A)}{2}$$
\end{thm}

\begin{proof}

By using the partition $\xi$ subordinated to local stable
manifolds from above, we can associate conditional measures of
$\mu_s$, denoted by $\mu_{s, A}^s$, $A \in \xi$. We want to
estimate the measure $\mu_{s, A}^s$ of a small arbitrary ball
$B(y, \rho)$ centered at some $y \in A$, where $W^s_{r(A)/2}(x)
\cap \Lambda \subset A \subset W^s_{r(A)}(x) \cap \Lambda, x =
x(A)$.

Let us first consider an arbitrary set $f^n(B_n(z, \vp))$, where
we remind that $B_n(z, \vp)$ denotes a Bowen ball, and where
$\vp>0$ is arbitrary but small. This set is actually a
neighbourhood of the unstable manifold $W^u_r(\hat f^n z)$
corresponding to a prehistory $(f^n z, f^{n-1}z, \ldots, z,
\ldots)$. We will estimate the $\mu_s$-measure of a cross section
of such a set $f^n(B_n(z, \vp))$, i.e an intersection of type
$$B(n, z; k, x; \vp):= f^n(B_n(z, \vp)) \cap B_k(x, \vp),$$ for
arbitrary $z, x \in \Lambda$ and positive integers $n, k$ . We see
that if we vary $z, x, k, n$, we can write any open set in
$\Lambda$, as a union of mutually disjoint sets of type $B(n, z;
k, x; \vp)$.

So let us estimate the $\mu_s$-measure of $B(n, z; k, x, \vp)$.
Notice that $B(n, z; k, x; \vp)$ is contained in $f^{n}(B_{n+k}(z,
\vp))$.  Without loss of generality we can assume that $z =
x_{-n}$, i.e that $z$ itself is the unique $n$-preimage of $x$
inside $B_n(z, \vp)$; if not, then we can replace $z$ by a point
$x_{-n}$ which is $\vp$-shadowed by $z$ up to order $n+k$, and
thus the dynamical behaviour of $z$ up to order $n+k$ will be the
same as that of $x_{-n}$.

Let us denote the positive quantity $|Df_s^n(z)|\cdot \vp$ by
$\rho$.  Since the endomorphism $f$ is conformal on local stable
manifolds, the diameter of the intersection $f^n(B_n(z, \vp)) \cap
W^s_r(f^nz)$ is equal to $2\rho$.

Now recall that we assumed without loss of generality that $f^nz =
x$, and consider all the finite prehistories of the point $x$, in
$\Lambda$. We will call then \textbf{$\rho$-maximal prehistory} of
$x$ any finite prehistory $(x, x_{-1}, \ldots, x_{-p})$ so that
$|Df^p(x_{-p+1})| \cdot \vp \ge \rho$ but $|Df^p(x_{-p})| \cdot
\vp < \rho$. Clearly, given any prehistory $\hat x = (x, x_{-1},
\ldots )$ of $x$, there exists some positive integer $n(\hat x,
\rho)$ such that $(x, x_{-1}, \ldots, x_{-n(\hat x, \rho)})$ is a
$\rho$-maximal prehistory. Let us denote by $$\mathcal{N}(x,
\rho):= \{n(\hat x, \rho), \hat x \ \text{prehistory of} \ x \
\text{from} \ \Lambda\}$$

We will consider now the various components of the $p$-preimages
of $B(n, z; k, x; \vp)$, when $p$ ranges in $\mathcal{N}(x,
\rho)$. We extended the stable diameter of $B(n, z; k, x; \vp)$ in
backward time until we reach a diameter of at most $\vp$. As the
maximum expansion in backward time is realized on the stable
manifolds (local inverse iterates contract all the unstable
directions), it follows that for any prehistory $\hat x$ of $x$,
there exists a component of $f^{-n(\hat x, \rho)}(B(n, z; k, x;
\vp))$ inside the Bowen ball $B_{n(\hat x, \rho)}(x_{-n(\hat x,
\rho)}, \vp)$; denote this component by $A(\hat x, \rho)$.
 We see
that all these components $A(\hat x, \rho)$ are mutually disjoint
if $\vp << \vp_0$, where $\vp_0$ is the local injectivity constant
of $f$ on $\Lambda$ (recall that there are no critical points in
$\Lambda$). Indeed if the sets $A(\hat x, \rho)$ and $A(\hat x',
\rho)$ would intersect for some prehistories $\hat x =(x, x_{-1},
\ldots) , \hat x' = (x, x'_{-1}, \ldots)$ of $x$ then, since they
are contained in Bowen balls, their forward iterates would be
$2\vp$-close. But then we get a contradiction since the
prehistories $\hat x, \hat x'$ must contain different preimages
$x_p, x-p'$ at some level $p$, and these different preimages must
be at a distance of at least $\vp_0$ from each other. Hence either
$A(\hat x, \rho) = A(\hat x', \rho)$, or $A(\hat x, \rho) \cap
A(\hat x', \rho) = \emptyset$.

Now we will use the $f$-invariance of the equilibrium measure
$\mu_s$ in order to estimate the $\mu_s$-measure of the set $B(n,
z; k, x; \vp)$. Recall that $f^n z = x$, and $\vp |Df_s^n(z)| =:
\rho$. Then we have $$\mu_s(B(n, z; k, x; \vp) =
\mathop{\sum}\limits_{\hat x \ \text{prehistory of} \ x}
\mu_s(A(\hat x, \rho)), $$ since we showed above that the sets
$A(\hat x, \rho)$ either coincide or are disjoint.

Now let us take two sets $A(\hat x, \rho), A(\hat x', \rho)$, one
of them with $n(\hat x, \rho) = p$ and the other with $n(\hat x',
\rho) = p'$.  We proved in \cite{MU-CJM} that for a $d$-to-1
c-hyperbolic endomorphism $f$ on the basic set $\Lambda$, we have
$\delta^s = t^s_d$, where $t^s_d$ is the unique zero of the
pressure function $t \to P(t\Phi^s - \log d)$. Therefore we can
use that
\begin{equation}\label{CJM}
P(\delta^s\Phi^s) = \log d
\end{equation}
 Then from the definition of the sets
$A(\hat x, \vp)$ and by using Proposition \ref{pieces}, we can
compare the measure $\mu_s$ on two sets $A(\hat x, \rho), A(\hat
x', \rho)$ as follows:
\begin{equation}\label{comp}
\frac{1}{C_\vp} \mu_s(A(\hat x', \rho))
\frac{|Df_s^p(x_{-p})|^{\delta^s}}{|Df_s^{p'}(x'_{-p'})|^{\delta^s}}
\cdot d^{p'-p} \le \mu_s(A(\hat x, \rho)) \le C_\vp \mu_s(A(\hat
x',
\rho))\frac{|Df_s^p(x_{-p})|^{\delta^s}}{|Df_s^{p'}(x'_{-p'})|^{\delta^s}}
\cdot d^{p'-p}
\end{equation}

In general, if for two variable quantities $Q_1, Q_2$, there
exists a positive universal constant $c$ such that $\frac{1}{c}Q_2
\le Q_1 \le cQ_2$, we say that $Q_1, Q_2$ are \textbf{comparable},
and will denote this by $Q_1 \approx Q_2$; the constant $c$ is
called the \textbf{comparability constant}.

But from the definition of $n(\hat x, \rho)$ above (as being the
length of the $\rho$-maximal prehistory along $\hat x$), and since
$n(\hat x, \rho) = p, n(\hat x', \rho) = p'$ we obtain:
$$\frac{1}{C} |Df_s^{p'}(x'_{-p'})| \le |Df_s^p(x_{-p})| \le C
|Df_s^{p'}(x'_{-p'})|$$ Therefore,  from relation (\ref{comp}) we
obtain
\begin{equation}\label{comp2}
\frac{1}{C_\vp} \mu_s(A(\hat x', \rho))d^{p'-p} \le \mu_s(A(\hat
x, \rho)) \le C_\vp \mu_s(A(\hat x', \rho))d^{p'-p},
\end{equation}
 where we used the same constant $C_\vp$ as in
(\ref{comp}), without loss of generality. Hence the proof will now
be reduced to a combinatorial argument about the different
pieces/components, of the preimages of various orders of $B(n, z;
k, x; \vp)$.

However we assumed that every point from $\Lambda$ has exactly $d$
$f$-preimages inside $\Lambda$. We use (\ref{comp2}) in order to
compare the $\mu_s$-measures of the different pieces $A(\hat x,
\rho)$, which will then be added successively. Recall that one of
these components $A(\hat x, \rho)$ is precisely $B_{n+k}(z, \vp)$.
The comparisons will always be made with respect to this component
$B_{n+k}(z, \vp)$. Let us order the integers from $\mathcal{N}(x,
\rho)$ as: $$n_1 > n_2 >\ldots > n_T$$ We shall add first the
measures $\mu_s(A(\hat x, \rho))$ over all the sets corresponding
to $\hat x$ with $n(\hat x, \rho) = n_1$, then over those
prehistories with $n(\hat x, \rho) = n_2$, etc. And will use that
any point from $\Lambda$ has exactly $d^m$ $m$-preimages belonging
to $\Lambda$ for any $m \ge 1$. \
 Therefore by such successive
additions and by using (\ref{comp2}) we obtain: $$\mu_s(B_{n+k}(z,
\vp)) \cdot d^{n} \le \mu_s(B(n, z; k, x; \vp)) =
\mathop{\sum}\limits_{\hat x \ \text{prehistory of} \ x}
\mu_s(A(\hat x, \rho)) \le \mu_s(B_{n+k}(z, \vp)) \cdot d^{n},$$
with the positive constant $C_\vp$ independent of $n, k, z, x$.

We use now Theorem 1 of \cite{M-DCDS06} which gave estimates for
equilibrium measures on Bowen balls, similar to those from the
case of diffeomorphisms (see \cite{KH} for example); this was done
by lifting to an equilibrium measure on $\hat \Lambda$. Hence from
the last displayed formula and (\ref{CJM}), we obtain:

\begin{equation}\label{tub}
\frac{1}{C_\vp} d^n \cdot
\frac{|Df_s^{n+k}(z)|^{\delta^s}}{d^{n+k}} \le \mu_s(B(n, z; k, x;
\vp)) \le C_\vp d^n \cdot
\frac{|Df_s^{n+k}(z)|^{\delta^s}}{d^{n+k}}
\end{equation}

Let us now study in more detail the conditions from the definition
of conditional measures. From the construction of the measurable
partition $\xi$ we have that $W^s_{r(A)/2}(x) \cap \Lambda \subset
A \subset W^s_{r(A)}(x) \cap \Lambda, x = x(A) \in S$ and the
radii $r(A)$ vary continuously with $A$. So from Remark 1 we can
split an arbitrary set $U \in \mathcal{P}$, modulo $\mu_s$, into a
disjoint union of open sets $V$, each being a $\xi$-set, so there
exists $r = r(V)>0$ s.t for all $A \in \xi$ intersecting $V$, we
have $W^s_{r/2}(x(A)) \cap \Lambda \subset A \subset W^s_r(x(A))
\cap \Lambda$.  Hence locally, on a subset $V \subset U \in
\mathcal{P}$, we can consider that $\xi$ is, modulo a set of
$\mu_s$-measure zero, a foliation with local stable manifolds
$W^s_r(x)$ of the same size $r=r(V)$. The intersections of these
local stable manifolds with $\Lambda$ are then identified with
points in the factor space $\Lambda/\xi$.

We will work for the rest of the proof on an open set $V$ as
above, i.e where the sets $A \in \xi$ can be assumed to be of type
$W^s_r(x)$, of the same size $r=r(V)$. Take also $\vp = r$.

From the definition of the factor space $\Lambda/\xi$, the
$(\mu_s)_\xi$-measure induced on the quotient space $\Lambda/\xi$
is given by $(\mu_s)_\xi(E) = \mu_s(\pi_\xi^{-1}(E))$, where
$\pi_\xi:\Lambda \to \Lambda/\xi$ is the canonical projection
which collapses a set from $\xi$ to a point. \ We notice that the
projection $\pi_\xi(B(n, z; k, x; r))$ in $\Lambda/\xi$ has
$(\mu_s)_\xi$-measure equal to $\mu_s(B_k(x, r))$, since
$\pi_\xi^{-1}(\pi_\xi(B(n, z; k, x; r))$ is $B_k(x, r)$. Now since
$P(\delta^s \Phi^s) = \log d$ (from relation (\ref{CJM})) and by
using again the estimates of equilibrium states on Bowen balls, we
obtain as in (\ref{tub}) that $\mu_s(B_k(x, r))$ is comparable to
$\frac{|Df_s^k(x)|^{\delta^s}}{d^k}$ (with a comparability
constant $c=c(V)$).

But,  from the definition of conditional measures we have

\begin{equation}\label{fam}
\mu_s(B(n, z; k, x; r)) = \int_{B_k(x, r)/\xi} \mu_{s, A}^s(A \cap
B(n, z; k, x; r)) d(\mu_s)_\xi(\pi_\xi(A))
\end{equation}

Now by (\ref{tub}) and recalling that $f^nz =x$, we infer that
$\mu_s(B(n, z; k, x; r))$ is comparable to
$\frac{|Df_s^k(x)|^{\delta^s}}{d^k} \cdot \rho^{\delta^s}$, i.e to
$\mu_s(B_k(x, r)) \cdot \rho^{\delta^s}$, where in our case
$\rho:= |Df_s^n(z)|r$ (the comparability constant being denoted by
$C_V$). And we showed above that $(\mu_s)_\xi(B_k(x, r)/\xi) =
\mu_s(B_k(x, r)) \approx \frac{|Df_s^k(x)|^{\delta^s}}{d^k}$ (with
the comparability constant $C_V$).

In addition, we notice that the sets of type $B(n, z; k, x; r)$
where $n, z, k, x$ vary, form a basis for the open sets in $V$;
also, if we vary $n$, the radius $\rho = |Df_s^n(z)|\cdot r$ can
be made arbitrarily small.  \ Therefore from the essential
uniqueness of the system of conditional measures associated to
$(\mu_s, \xi)$ and since any borelian set in $V$ can be written
modulo $\mu_s$ as a union of disjoint sets of type $B(n, z; k, x;
r)$, we conclude that the conditional measure $\mu_{s, A}^s$ is a
geometric probability of exponent $\delta^s$. Hence for all
$\rho$, $0< \rho < r/2$, we have
 $$\frac{1}{C_V} \rho^{\delta^s} \le
\mu_{s, A}^s(B(y, \rho)) \le C_V \rho^{\delta^s}, y \in A,$$ for
$A \subset V, A \in \xi$. The comparability factor $C_V$ is
constant on $V$; in general it can be taken locally constant on
the complement in $\Lambda$ of a set of $\mu_s$-measure zero. The
proof is thus finished.
\end{proof}

\begin{defn}\label{stable-cond}
Let $f$ be a hyperbolic endomorphism on the folded basic set
$\Lambda$, $\mu$ a borelian probability measure on $\Lambda$ and
$\xi$ a measurable partition subordinated to local stable
manifolds. Then the conditional measure $\mu_A^s$ corresponding to
$A \in \xi$ will be called \textit{the stable conditional measure}
of $\mu$ on $A$. When $\mu = \mu_s$ we denote this stable
conditional measure by $\mu^s_{s, A}$.
\end{defn}

\textbf{Remark 2.} We notice from the proof of Theorem \ref{main}
that, in fact, the stable conditional measures of $\mu_s$ do not
depend on the measurable partition $\xi$ constructed above,
subordinated to local stable manifolds. Therefore there exists a
set $\Lambda(\mu_s)$ of full $\mu_s$-measure inside $\Lambda$,
such that for every $x \in \Lambda(\mu_s)$ there exists some small
$r(x)>0$ so that $W^s_{r(x)}(x)$ is contained in a set $A$ from a
measurable partition of type $\xi$ (subordinated to local stable
manifolds); then one can construct the stable conditional measure
$\mu^s_{s, A}$. We denote this conditional measure also by
$\mu^s_{s, x}, x \in \Lambda(\mu_s)$. $\hfill\square$

We recall now the notions of \textit{lower}, respectively
\textit{upper pointwise dimension} of a finite borelian measure
$\mu$ on a compact space $\Lambda$ (see for example \cite{Ba}).
For $x \in \Lambda$, they are defined by $$ \underline{d}_\mu(x):=
\mathop{\liminf}\limits_{\rho \to 0} \frac{\log \mu(B(x,
\rho))}{\log \rho}, \ \text{and} \ \bar{d}_\mu(x):=
\mathop{\limsup}\limits_{\rho \to 0} \frac{\log \mu(B(x,
\rho))}{\log \rho}$$ If the lower pointwise dimension at $x$
coincides with the upper pointwise dimension at $x$, we denote the
common value by $d_\mu(x)$ and call it simply the
\textit{pointwise dimension} at $x$.

 One can also define the \textit{Hausdorff
dimension, lower box dimension and upper box dimension} of $\mu$
respectively by: $$ HD(\mu):= \inf\{HD(Z), \mu(\Lambda \setminus
Z) = 0\}$$ $$\underline{\text{dim}}_B(\mu):= \lim_{\delta \to 0}
\inf\{\underline{\text{dim}}_B(Z), \mu(\Lambda \setminus Z) \le
\delta\}$$ $$\overline{\text{dim}}_B(\mu):=\lim_{\delta\to 0} \inf
\{\overline{\text{dim}}_B(Z), \mu(\Lambda\setminus Z) \le
\delta\}$$

Assume now in general that $f$ is a hyperbolic endomorphism on
$\Lambda$ and $\mu$ a probability measure on $\Lambda$, and let
$\xi$ be a measurable partition subordinated to local stable
manifolds of $f$ on $\Lambda$. We define then the
\textit{lower/upper stable pointwise dimension} of $\mu$ at $y$,
for $\mu$-a.e $y \in \Lambda$, as the lower/upper pointwise
dimension of the stable conditional measure $\mu^s_A$ at $y$, for
$y \in A$, namely: $$ \underline{d}_\mu^s(y):=
\mathop{\liminf}\limits_{\rho \to 0} \frac{\log \mu_A^s(B(y,
\rho))}{\log \rho} \ \ \text{and} \ \bar{d}_\mu^s(y):=
\mathop{\limsup}\limits_{\rho \to 0} \frac{\log \mu_A^s(B(y,
\rho))}{\log \rho}$$

Similarly we define the \textit{stable Hausdorff dimension} of
$\mu$ on $A \in \xi$, and the \textit{stable lower/upper box
dimension} of $\mu$ on $A$, respectively, as the quantities:
$$HD^s(\mu, A):= HD(\mu_A^s), \  \ \underline{\text{dim}}_B^s(\mu,
A) := \underline{\text{dim}}_B(\mu_A^s), \  \
\overline{\text{dim}}_B^s(\mu, A) :=
\overline{\text{dim}}_B(\mu_A^s), A \in \xi$$

When $\mu = \mu_s$ we denote $HD^s(\mu_s, x) := HD(\mu^s_{s, x})$,
$\underline{\text{dim}}_B^s(\mu_s, x) :=
\underline{\text{dim}}_B(\mu^s_{s, x})$, and
$\overline{\text{dim}}_B^s(\mu_s, x) :=
\overline{\text{dim}}_B(\mu_{s, x}^s)$, for $x \in
\Lambda(\mu_s)$.

Recall now the stable dimension $\delta^s$ from Definition
\ref{stable} and the Theorem of Independence of the Stable
Dimension given afterwards.

\begin{cor}\label{pointwise}
Let $f$ be a c-hyperbolic, $d$-to-1 endomorphism on a basic set
$\Lambda$, and $\mu_s$ be the equilibrium measure of the potential
$\delta^s \Phi^s$. Then the stable pointwise dimension of $\mu_s$
exists $\mu_s$-almost everywhere on $\Lambda$ and is equal to the
stable dimension $\delta^s$.

Also the stable Hausdorff dimension of $\mu_s$, stable lower box
dimension of $\mu_s$ and stable upper box dimension of $\mu_s$ are
all equal to $\delta^s$.
\end{cor}

\begin{proof}
The proof follows from Theorem \ref{main} since we proved that the
stable conditional measures of the equilibrium measure $\mu_s$ are
geometric probabilities.

For the second part of the Corollary, we use Theorem 2.1.6 of
\cite{Ba}. Indeed since the stable conditional measures of $\mu_s$
are geometric probabilities of exponent $\delta^s$, we conclude
that the stable Hausdorff, lower/upper dimensions coincide, and
are all equal to the stable dimension $\delta^s$.
\end{proof}

\begin{defn}\label{msd}
We will say that a measure $\mu$ on $\Lambda$ has \textit{maximal
stable dimension} on $A \in \xi, A \subset W^s_{r(x)}(x)$ if:
$$HD^s(\mu, A) = \sup\{HD^s(\nu, A), \nu \ \text{is an} \
f|_\Lambda-\text{invariant probability measure} \ \text{on} \
\Lambda\}$$
\end{defn}

This definition is similar to that of measure of maximal
dimension; see \cite{Ba}, \cite{BW} where measures of maximal
dimension on hyperbolic sets of surface diffeomorphisms were
studied. Our setting/methods for the maximal \textit{stable}
dimension in the non-invertible case, are however different. \

Now, since the stable Hausdorff dimension of any $f$-invariant
probability measure $\nu$ on $\Lambda$ is bounded above by
$\delta^s := HD(W^s_r(x) \cap \Lambda)$, we see from Corollary
\ref{pointwise} that:

\begin{cor}\label{maximal}
In the setting of Theorem \ref{main} it follows that the stable
equilibrium measure $\mu_s$ of $f$, is of \textbf{maximal stable
dimension} on $W^s_{r(x)}(x) \cap \Lambda$ among all $f$-invariant
probability measures on $\Lambda$, for $\mu_s$-a.e $x \in
\Lambda$. And $\mu_s$ maximizes in a Variational Principle for
stable dimension on $\Lambda$, i.e: $$\delta^s = HD^s(\mu_s, x) =
\sup\{HD^s(\nu, x),  \ \nu \ \text{is an} \
f|_\Lambda-\text{invariant probability measure} \ \text{on} \
\Lambda\}, \mu_s-a.e \ x$$
\end{cor}

 We say now that the basic set $\Lambda$ is a
\textbf{repellor} (or \textbf{folded repellor}) if there exists a
neighbourhood $U$ of $\Lambda$ such that $\bar U \subset f(U)$.
And that $\Lambda$ is a \textit{local repellor} if there are local
stable manifolds of $f$ contained inside $\Lambda$ (see
\cite{M-Cam} for more on these notions in the case of
endomorphisms).

\begin{cor}\label{absolute}
Let an open c-hyperbolic endomorphism $f$ on a connected basic set
$\Lambda$. Then we have that the stable conditional measures
$\mu_{s, x}^s$ of $\mu_s$, are absolutely continuous with respect
to the induced Lebesgue measures on $W^s_{r(x)}(x), x \in
\Lambda(\mu_s)$, if and only if $\Lambda$ is a non-invertible
repellor.
\end{cor}

\begin{proof}
If $f$ is open on a connected $\Lambda$ we saw in Section 1 that
$f$ is constant-to-1 on $\Lambda$.

 The first part of the proof follows
from Theorem \ref{main} and from Theorem 1 of \cite{M-Cam}. Indeed
in \cite{M-Cam} we showed that in the above setting, if none of
the stable manifolds centered at $x$ is contained in $\Lambda$,
then $\delta^s$ is strictly less than the real dimension $d_s$ of
the manifold $W^s_{r(x)}(x)$ (the result in \cite{M-Cam}, given
for the case when $d_s$ is 2, can be generalized easily to other
dimensions as long as the condition of conformality on stable
manifolds is satisfied).  Thus in order to have absolute
continuity of the stable conditional measures we must have some
local stable manifolds contained in $\Lambda$, equivalent to
$\Lambda$ being a local repellor (in the terminology of
\cite{M-Cam}). But we proved in Proposition 1 of \cite{M-Cam} that
when $f|_\Lambda: \Lambda \to \Lambda$ is open, then $\Lambda$ is
a local repellor if and only if $\Lambda$ is a repellor.

The converse is clearly true since, if $\Lambda$ is a repellor,
then the local stable manifolds are contained inside $\Lambda$,
and thus the stable dimension $\delta^s$ is equal to the dimension
$d_s$ of the manifold $W^s_{r(x)}(x)$. Hence from Theorem
\ref{main} it follows that the stable conditional measures of
$\mu_s$ are geometric of exponent $d_s$; thus they are absolutely
continuous with respect to the respective induced Lebesgue
measures.
\end{proof}

Let us give in the end some examples of c-hyperbolic endomorphisms
which are constant-to-1 on basic sets, for which we will apply
Theorem \ref{main} and its Corollaries.

 \textbf{Example 1.} The first and simplest example is that of a product $$f(z, w) =
(f_1(z), f_2(w)),(z, w) \in \mathbb C^2$$ where $f_1$ has a fixed
attracting point $p$ and $f_2$ is expanding on a compact invariant
set $J$. Then the basic set that we consider is $\Lambda:=
\{p\}\times J$. For instance take $f(z, w) = (z^2+c, w^2), c \ne
0$, $|c|$ small, on the basic set $\Lambda = \{p_c\}\times S^1$,
where $p_c$ denotes the unique fixed attracting point of $z \to
z^2+c$. The stable dimension here is equal to zero and the
intersections of type $W^s_r(x) \cap \Lambda$ are singletons.

\textbf{Example 2.} We can take a hyperbolic toral endomorphism
$f_A$ on $\mathbb T^2$, where $A$ is an integer-valued matrix with
one eigenvalue of absolute values strictly less than 1, and
another eigenvalue of absolute value strictly larger than 1. In
this case we can take $\Lambda = \mathbb T^2$, and we have the
stable dimension equal to 1. We see that $f_A$ is
$|\text{det}(A)|$-to-1 on $\mathbb T^2$.

We may take also $f_{A, \vp}$ a perturbation of $f_A$ on $\mathbb
T^2$. Then again $f_{A, \vp}$ is $|\text{det}(A)|$-to-1 on
$\mathbb T^2$, and c-hyperbolic on $\mathbb T^2$. The stable
dimension is equal to 1, but the stable potential $\Phi^s$ is not
necessarily constant now. From Corollary \ref{absolute} we see
that the stable conditional measures of the equilibrium measure
$\mu_s$ are absolutely continuous.

\textbf{Example 3.} We construct now examples of folded repellors
which are not necessarily Anosov endomorphisms.

We remark first that if $\Lambda$ is a repellor for an
endomorphism $f$, with neighbourhood $U$ so that $\bar U \subset
f(U)$, then $f^{-1}(\Lambda) \cap U = \Lambda$. Therefore if
$\Lambda$ is in addition connected, it follows easily that $f$ is
constant-to-1 on $\Lambda$.  \  \ \ Let us show now that
constant-to-1 repellors are stable under perturbations.

\begin{pro}\label{stability}
Let $\Lambda$ be a connected repellor for an endomorphism $f$ so
that $f$ is hyperbolic on $\Lambda$, and let a perturbation
$f_\vp$ which is  $\mathcal{C}^1$-close to $f$. Then $f_\vp$ has a
connected repellor $\Lambda_\vp$ close to $\Lambda$, and such that
$f_\vp$ is hyperbolic on $\Lambda_\vp$. Moreover for any $x \in
\Lambda_\vp$,  the number of $f_\vp$-preimages of $x$ belonging to
$\Lambda_\vp$, is the same as the number of $f$-preimages in
$\Lambda$ of a point from $\Lambda$.
\end{pro}

\begin{proof}

Since $\Lambda$ has a neighbourhood $U$ so that $\bar U \subset
f(U)$, it follows that for $f_\vp$ close enough to $f$, we will
obtain $\bar U \subset f_\vp(U)$. If $f_\vp$ is
$\mathcal{C}^1$-close to $f$, then we can take the set $
\Lambda_\vp:= \mathop{\cap}\limits_{n \in \mathbb Z} f_\vp^n(U)$,
and it is quite well-known that $f_\vp$ is hyperbolic on
$\Lambda_\vp$ (for example \cite{Ru-carte},  etc.)

We know that there exists a conjugating homeomorphism $H: \hat
\Lambda \to \hat \Lambda_\vp$ which commutes with $\hat f$ and
$\hat f_\vp$. The natural extension $\hat \Lambda$ is connected
iff $\Lambda$ is connected. Hence $\hat \Lambda_\vp$ is connected
and so $\Lambda_\vp$ is also connected. Moreover since $\bar U
\subset f_\vp(U)$, we obtain that $\Lambda_\vp$ is a connected
repellor for $f_\vp$.

Now assume that $x \in \Lambda$ has $d$ $f$-preimages in
$\Lambda$. Then if $C_f \cap \Lambda = \emptyset$ and if $f_\vp$
is $\mathcal{C}^1$-close enough to $f$, it follows that the local
inverse branches of $f_\vp$ are close to the local inverse
branches of $f$ near $\Lambda$. Therefore any point $y \in
\Lambda_\vp$ has exactly $d$ $f_\vp$-preimages in $U$, denoted by
$y_1, \ldots, y_d$. Any of these $f_\vp$-preimages from $U$ has
also an $f_\vp$-preimage in $U$ since $\bar U \subset f_\vp(U)$,
etc. Thus $y_i \in \Lambda_\vp = \mathop{\cap}\limits_{n \in
\mathbb Z} f_\vp^n(U), i = 1, \ldots, d$; \ hence any point $y \in
\Lambda_\vp$ has exactly $d$ $f_\vp$-preimages belonging to the
repellor $\Lambda_\vp$.
\end{proof}

 Let us now take the
hyperbolic toral endomorphism $f_A$ from Example 2, and the
product $f(z, w) = (z^k, f_A(w)), (z, w) \in \mathbb P^1\mathbb C
\times \mathbb T^2$, for some fixed $k \ge 2$. \ And consider a
$\mathcal{C}^1$-\textbf{perturbation} $f_\vp$ of $f$ on $\mathbb
P^1 \mathbb C \times \mathbb T^2$. \ \  Since $f$ is c-hyperbolic
on its connected repellor $\Lambda:= S^1 \times \mathbb T^2$, it
follows from Proposition \ref{stability} that the perturbation
$f_\vp$ also has a connected folded repellor $\Lambda_\vp$, on
which it is c-hyperbolic. Also it follows from above that $f_\vp$
is constant-to-1 on $\Lambda_\vp$, namely it is
$(k+|\text{det}(A)|)$-to-1. The stable dimension $\delta^s(f_\vp)$
of $f_\vp$ on $\Lambda_\vp$ is equal to 1 in this case. We can
form the stable potential of $f_\vp$, namely $\Phi^s(f_\vp)(z, w)
:= \log |D(f_\vp)_s|(z, w), (z, w) \in \Lambda_\vp$, and the
equilibrium measure $\mu_s(f_\vp)$ of $\delta^s(f_\vp) \cdot
\Phi^s(f_\vp)$, like in Theorem \ref{main}. Since the basic set
$\Lambda_\vp$ is a repellor, we obtain from Corollary
\ref{absolute} that the stable conditional measures of
$\mu_s(f_\vp)$ are absolutely continuous on the local stable
manifolds of $f_\vp$ (which in general are non-linear
submanifolds).

One actual example can be constructed by the above procedure, if
we consider first the linear toral endomorphism $f_A(w) =
(3w_1+2w_2, 2w_1+2w_2), w = (w_1, w_2) \in \mathbb R^2/\mathbb
Z^2$. The associated matrix $A$ has one eigenvalue of absolute
value less than 1 and the other eigenvalue larger than 1, hence
$f_A$ is hyperbolic on $\mathbb T^2$. And as above we can take the
product $f(z, w) = (z^k, f_A(w))$ for some $k \ge 2$. \ Then we
consider the perturbation endomorphism: $$f_\vp(z, w) := (z^k,
3w_1+2w_2+\vp sin(2\pi(w_1+5w_2)), 2w_1+2w_2+\vp cos(2\pi w_2)
+\vp sin^2(\pi(w_1-2w_2))),$$ defined for $z \in \mathbb P^1
\mathbb C, w \in \mathbb T^2$.
 We see that $f_\vp$ is
well defined as an endomorphism on $\mathbb P^1 \mathbb C \times
\mathbb T^2$ and that it has a repellor $\Lambda_\vp$ close to
$S^1 \times \mathbb T^2$, given by Proposition \ref{stability};
namely there exists a neighbourhood $U$ of $S^1 \times \mathbb
T^2$ so that $$\Lambda_\vp = \mathop{\cap}\limits_{n \in \mathbb
Z} f_\vp^n(U)$$
 Then $f_\vp$ is c-hyperbolic on $\Lambda_\vp$ (see
Definition \ref{stable}) and it is $(k+2)$-to-1 on $\Lambda_\vp$.
The stable potential $\Phi^s(f_\vp)$ is not necessarily constant
in this case. We obtain as before that the stable conditional
measures of $\mu_s(f_\vp)$ are absolutely continuous, and that the
stable pointwise dimension of $\mu_s(f_\vp)$ is essentially equal
to 1, on $\mu_s$-a.a local stable manifolds over $\Lambda_\vp$.
$\hfill\square$

 \textbf{Acknowledgements:} Partial support
for this work was provided by PN II Project ID-1191.

\textbf{Email:}  Eugen.Mihailescu\@@imar.ro

Institute of Mathematics of the Romanian Academy, P. O. Box 1-764,
RO 014700, Bucharest, Romania.

 Webpage: www.imar.ro/$\sim$mihailes

\end{document}